\documentclass[12pt]{article}
\usepackage{amssymb,amsmath,amsfonts}
\usepackage{amsthm}
\usepackage{multirow}
\usepackage{epsfig}
\usepackage{subfigure}
\usepackage{color}
\usepackage{slashbox}

\usepackage{mathabx}

\numberwithin{equation}{section}
\theoremstyle{plain}
\newtheorem{theorem}{Theorem}[section]
\newtheorem{corollary}[theorem]{Corollary}
\newtheorem{proposition}[theorem]{Proposition}
\newtheorem{lemma}[theorem]{Lemma}
\newtheorem{observation}[theorem]{Observation}
\theoremstyle{definition}
\newtheorem{definition}{Definition}[section]

\theoremstyle{remark}
\newtheorem{remark}{\rm\bf Remark}[section]

\allowdisplaybreaks

\begin{document}

\title{Uncorrelatedness sets of discrete uniform distributions via Vandermonde-type determinants}
\author{Mehmet Turan$^1$, Sofiya Ostrovska$^1$ and Ahmet Ya\c{s}ar \"{Ozban}$^2$}
\date{}
\maketitle

\begin{center}
{\it $^1$ Atilim University, Department of Mathematics, 06836, Ankara, Turkey}\\
{\it $^2$ Karatekin University, Department of Mathematics, Cankiri, Turkey}\\
{\it e-mail: mehmet.turan@atilim.edu.tr, sofia.ostrovska@atilim.edu.tr, ahmetyozban@karatekin.edu.tr}%\\
%{\it Tel: +90 312 586 8211,  Fax: +90 312 586 8091}
\end{center}

\begin{abstract}
Given random variables $X$ and $Y$ having finite moments of all orders, their uncorrelatedness set is defined as the set of all pairs $(j,k)\in{\mathbb N}^2,$ for which $X^j$ and $Y^k$ are uncorrelated. It is known that, broadly put, 
any subset of ${\mathbb N}^2$ can serve as an uncorrelatedness set. This claim ceases to be true for random variables with prescribed distributions, in which case the need arises so as to identify the admissible uncorrelatedness sets. This paper studies the uncorrelatedness sets for positive random variables uniformly distributed on three points. Some general features of these sets are derived. Two related Vandermonde-type determinants are examined and applied to describe uncorrelatedness sets in specific cases.
\end{abstract}

{\bf Keywords}: Uncorrelatedness set, random variable, discrete uniform distribution, determinant

{\bf 2010 MSC:} 
60E05, %Distributions: general theory
62H20, %Measures of association
15A15 %Determinants, permanents, other special matrix functions
\section{Introduction and preliminaries}

Since the concept of independence is fundamental in Probability Theory, Mathematical Statistics and their applications, various notions have been developed related to the independence of random variables. 
See, for example \cite[Sections 3 and 7]{counter}. 
The earliest of such notions are uncorrelatedness and correlation coefficient, both widely used in statistical analysis. For a brief history and their relation to the independence property, see \cite{davidedwards, david}.
An extension of the uncorrelatedness property for the powers of random variables lead to the following definition.
\begin{definition}\cite{proc}
Let $(X, Y)$ be a random vector whose components, $X$ and $Y,$ possess finite moments of all orders. The set 
$${\mathcal U}(X,Y)=\{(j,k)\in{\mathbb N}^2: {\bf E}[X^jY^k]={\bf E}[X^j]\cdot{\bf E}[Y^k]\}$$
is called an {\it uncorrelatedness set} of $X$ and $Y.$
\end{definition}
Clearly, random variables $X$ and $Y$ are {\it uncorrelated} if and only if $(1,1)\in{\mathcal U}(X,Y).$ For independent random variables ${\mathcal U}(X,Y)={\mathbb N}^2,$ while the converse, principally, is not true. To some extent, uncorrelatedness sets provide a partial order for degrees of independence: One may say that the wider an uncorrelatedness set is, the more independent the random variables are. It has to be pointed out that the independence of degree $k,$ defined by 
C.M. Cuadras in \cite{Cuadras}, is stated in terms of uncorrelatedness sets as follows: $X$ and $Y$ are independent of degree $k$ if and only if ${\{(j,l):j+l\leqslant k+1\}\subseteq \mathcal U}(X,Y).$ Uncorrelatedness sets are also related to The Italian Problem proposed by J. Stoyanov, see \cite{stoyanov}.

It is known that, in general, for any ${\mathcal A}\subseteq {\mathbb N}^2,$ there exist $X$ and $Y$ such that ${\mathcal U}(X,Y)={\mathcal A}.$ However, the situation changes when $X$ and $Y$ have pre-assigned distributions. More explicitly, when the distributions of $X$ and $Y$ are fixed, not every subset of ${\mathbb N}^2$ can serve as an uncorrelatedness set. In such a case, the problem arises as to finding admissible uncorrelatedness sets. For random variables with absolutely continuous distributions, this problem has been considered in \cite{proc, sofiyajmaa}.

In this work, some properties of uncorrelatedness sets of random variables having discrete uniform distributions are studied alongside possible uncorrelatedness sets ranging from the empty set $\emptyset$ to ${\mathbb N}^2.$ In addition, some related problems on determinants are considered. 
More precisely, let $(X,Y)$ be a random vector, whose marginals are uniformly distributed 
on the set $\{a,b,c\},$ where $0<a<b<c.$ The notation $X\sim \textnormal{Unif}\{a,b,c\}$ means that $X$ is uniformly distributed on the set $\{a,b,c\}.$
The joint probability mass function of $X$ and $Y$ can be expressed in the form given in Table \ref{lab:pmf},
where $x_i,$ $i=1,2,3,4$ are such that each entry in the table is non-negative.
\begin{table}[htbp]
\centering
\caption{Joint probability mass function of $X$ and $Y$} \label{lab:pmf}
\begin{tabular}{|c|ccc|}\hline
\backslashbox{$Y$}{$X$}
& $a$ & $b$ & $c$ \\ \hline & & & \\
$a$ & $\frac19+x_4$      & $\frac19+x_3$     & $\frac19-x_3-x_4$ \\ & & & \\
$b$ & $\frac19+x_2$      & $\frac19+x_1$     & $\frac19-x_1-x_2$ \\ & & & \\
$c$ & $\frac19-x_2-x_4$  & $\frac19-x_1-x_3$ & $\frac19+x_1+x_2+x_3+x_4$ \\ & & & \\ \hline
\end{tabular}
\end{table}

The condition $(j,k)\in{\mathcal U}(X,Y)$ is now equivalent to
\begin{align}\label{syst}
x_1+A_jx_2+A_kx_3+A_jA_kx_4=0,
\end{align}
where $A_j=(c^j-a^j)/(c^j-b^j),$ $j\in{\mathbb N}.$
Due to the condition $0<a<b<c,$ the sequence $\{A_j\}$ is strictly decreasing.
Notice that ${\mathcal A}$ is an uncorrelatedness set for $X$ and $Y$ if and only if system \eqref{syst} is satisfied solely for $(j,k)\in {\mathcal A}$ and violated for $(j,k)\notin {\mathcal A}.$ 
\begin{remark}\label{sca}
Since the system is homogeneous, every non-trivial solution can be re-scaled in such a way that the corresponding entries in 
Table \ref{lab:pmf} become non-negative. Therefore, in the sequel we only distinguish trivial and non-trivial solutions.
\end{remark}
\begin{remark} \label{rem:W}
It is not difficult to see from Table \ref{lab:pmf} and equation \eqref{syst} that, if ${\mathcal A}\subseteq {\mathbb N}^2$ is an admissible uncorrelatedness set for $X,Y\sim \textnormal{Unif}\{a,b,c\},$ then so is $\tilde{\mathcal A}=\{(j,k): (k,j)\in{\mathcal A}\}.$
\end{remark}
The paper is organized as follows: Section \ref{sec:sets} contains a series of uncorrelatedness sets along with properties. 
Section \ref{sec:det} deals with the determinants in the study. Finally, Section \ref{sec:op} presents the concluding remarks and proposes an open problem.

\section{Properties of uncorrelatedness sets}\label{sec:sets}

To begin with, let us notice that for random variables uniformly distributed on three points an uncorrelatedness set may be empty.

\begin{observation}
There exist $X,Y\sim \textnormal{Unif}\{a,b,c\}$ such that ${\mathcal U}(X,Y)=\emptyset.$
\end{observation}
\begin{proof}
Take $x_1=1$ and $x_i=0$ for $i=2,3,4.$ Then, the equality
$$x_1+A_j x_2+A_k x_3+A_jA_kx_4=0$$
is never satisfied for $P(j,k)\in{\mathbb N}^2.$ By Remark \ref{sca}, the result holds.
\end{proof}
It should be mentioned that this observation is valid not only for $\textnormal{Unif}\{a,b,c\},$ but also for any distribution with finite support.
\begin{proposition}\label{prop1}
For any $P_0(j_0,k_0)\in{\mathbb N}^2,$ there exist $X,Y\sim \textnormal{Unif}\{a,b,c\}$ such that ${\mathcal U}(X,Y)=\{P_0\}.$
\end{proposition}
\begin{proof}
We have to show that there is a 4-tuple $(x_1, x_2, x_3, x_4)$ so that
\begin{align*}
x_1 + A_{j}x_2 + A_{k}x_3 + A_{j}A_{k}x_4=0 \Leftrightarrow (j,k)=(j_0, k_0).
\end{align*}
Clearly, for the 4-tuple $(A_{j_0}+\sqrt{2}A_{k_0},-1,-\sqrt{2},0),$ one has
\begin{align}\label{jk0}
x_1+A_jx_2 + A_kx_3 + A_{j}A_{k}x_4=A_{j_0}-A_j+\sqrt{2}(A_{k_0}-A_k).
\end{align}
As $\{A_j\}$ is strictly decreasing, \eqref{jk0} vanishes if and only if $P=P_0.$
\end{proof}

\begin{corollary} The uncorrelatedness of $X$ and $Y$ does not imply their independence.
\end{corollary}

As for two-point sets, not all such sets can be uncorrelatedness sets. The next statement shows that a two-point set can serve as an uncorrelatedness set if and only if the points are not on the same horizontal or vertical line. In the latter case, the entire line is contained in ${\mathcal U}(X,Y).$

\begin{theorem} \label{thmhv}
Let $X,Y\sim \textnormal{Unif}\{a,b,c\}$ and $\{P_1, P_2\}\subseteq {\mathcal U}(X,Y),$ where $P_1(j_1, k_1)\neq P_2(j_2, k_2).$
\begin{enumerate}
\item [(i)] If $j_1=j_2=j,$ then $v_j:=\{(j,m):m\in{\mathbb N}\} \subseteq {\mathcal U}(X,Y).$
\item [(ii)] If $k_1=k_2=k,$ then $h_k:=\{(n,k):n\in{\mathbb N}\} \subseteq {\mathcal U}(X,Y).$
\item [(iii)] If $j_1\neq j_2$ and $k_1\neq k_2,$ then there exist $X,Y\sim \textnormal{Unif}\{a,b,c\}$ such that ${\mathcal U}(X,Y)=\{P_1, P_2\}.$
\end{enumerate}
\end{theorem}
\begin{proof}
\begin{enumerate}
\item [(i)]
Let $(x_1, x_2, x_3, x_4)$ be a solution of
\begin{align*}
\begin{array}{l}
x_1 + A_{j}x_2 + A_{k_1}x_3 + A_{j}A_{k_1}x_4 =0 \\
x_1 + A_{j}x_2 + A_{k_2}x_3 + A_{j}A_{k_2}x_4 =0.
\end{array}
\end{align*}
Then, 
$x_1+A_j x_2=0$ and $x_3+A_jx_4=0.$
Thus, for any $m\in{\mathbb N},$
one has
$$
x_1 + A_{j}x_2 + A_{m}x_3 + A_{j}A_{m}x_4 = x_1+A_jx_2+A_m(x_3+A_jx_4)=0
$$
which completes the proof.
\item [(ii)] The same as (i). 
\item [(iii)]
Equivalently, it has to be proved that there exists $(x_1, x_2, x_3, x_4)$ such that
\begin{align}\label{jk12}
\begin{array}{l}
x_1 + A_{j_1}x_2 + A_{k_1}x_3 + A_{j_1}A_{k_1}x_4=0 \\
x_1 + A_{j_2}x_2 + A_{k_2}x_3 + A_{j_2}A_{k_2}x_4=0
\end{array}
\end{align}
and 
\begin{align*}
x_1 + A_{j}x_2 + A_{k}x_3 + A_{j}A_{k}x_4\neq 0
\end{align*}
for all $P\notin \{P_1, P_2\}.$
System \eqref{jk12} has a 2-parameter family of solutions forming a two-dimensional subspace $V\subseteq {\mathbb R}^4.$
As the sequence $\{A_{j}\}$ is decreasing, it can be shown that 
$$
\text{rank}\begin{bmatrix}
1 & A_{j_1} & A_{k_1} & A_{j_1}A_{k_1} \\
1 & A_{j_2} & A_{k_2} & A_{j_2}A_{k_2} \\
1 & A_{j  } & A_{k  } & A_{j  }A_{k  }
\end{bmatrix} = 3,
$$
whence the system
\begin{align*}
x_1 + A_{j_1}x_2 + A_{k_1}x_3 + A_{j_1}A_{k_1}x_4 &=0 \\
x_1 + A_{j_2}x_2 + A_{k_2}x_3 + A_{j_2}A_{k_2}x_4 &=0 \\
x_1 + A_{j  }x_2 + A_{k  }x_3 + A_{j  }A_{k  }x_4 &=0 
\end{align*}
has a one-parameter set of solutions $\ell_P \subseteq V.$
Since there is only a countable number of straight lines $\ell_P,$ one has
$
\bigcup_{P\not\in U} \ell_P \neq V.
$
As a result, there is $(x_1, x_2, x_3, x_4)\in V\setminus \bigcup_{P\not\in U}.$
\end{enumerate}
\end{proof}

\begin{corollary}\label{cor:lp}
If $v_j \cup \{(l,k)\}\subseteq {\mathcal U}(X,Y),$ then $v_j \cup h_k \subseteq {\mathcal U}(X,Y).$ The same is true for a horizontal line and a point.
\end{corollary}

Now, let us examine feasible uncorrelatedness sets containing horizontal or vertical lines. The statement below elaborates assertions (i) and (ii) of Theorem \ref{thmhv} by showing that any single horizontal (vertical) line can constitute an uncorrelatedness set. However, if an uncorrelatedness set contains a horizontal (vertical) line $\ell$ and a point $P$ outside of that line, then it necessarily contains the line through $P$ perpendicular to $\ell.$ Further, the union of a horizontal and a vertical line can form an uncorrelatedness set.

\begin{theorem} \label{thm3p}
\begin{enumerate} 
\item [(i)] For each $j\in{\mathbb N},$ there exist $X,Y\sim \textnormal{Unif}\{a,b,c\}$ such that 
${\mathcal U}(X,Y)=v_j.$
\item [(ii)] For each $k\in{\mathbb N},$ there exist $X,Y\sim \textnormal{Unif}\{a,b,c\}$ such that 
${\mathcal U}(X,Y)=h_k.$
\item [(iii)] For each $(j,k)\in{\mathbb N}^2,$ there exist $X,Y\sim \textnormal{Unif}\{a,b,c\}$ such that 
${\mathcal U}(X,Y)=v_j \cup h_k.$
\end{enumerate}
\end{theorem}
\begin{proof}
\begin{enumerate}
\item [(i)] Let $x_1=x_2=0,$ $x_3=-A_j$ and $x_4=1.$ Then,
the system
\begin{align*}
x_1 + A_{j}x_2 + A_{k}x_3 + A_{j}A_{k}x_4&=0 \\
x_1 + A_{j}x_2 + A_{l}x_3 + A_{j}A_{l}x_4&=0
\end{align*}
is satisfied for all $(j,k), (j,l)\in v_j.$ However, when $m\neq j,$ one has
\begin{align*}
x_1 + A_{m}x_2 + A_{n}x_3 + A_{m}A_{n}x_4&=A_n(A_m-A_j) \neq 0.
\end{align*}
So, the line $v_j:=\{(j, k):k\in{\mathbb N}\}$ is an uncorrelatedness set.
\item [(ii)] The same as (i).
\item [(iii)]
Take $(m,l)\notin v_j \cup h_k.$
Assume that $(x_1, x_2, x_3, x_4)$  is a non-trivial solution of 
\begin{align*}
x_1 + A_{j}x_2 + A_{k}x_3 + A_{j}A_{k}x_4&=0\\
x_1 + A_{j}x_2 + A_{l}x_3 + A_{j}A_{l}x_4&=0\\
x_1 + A_{m}x_2 + A_{k}x_3 + A_{m}A_{k}x_4&=0.
\end{align*}
Then, by Theorem \ref{thmhv} (i) and (ii), both $v_m$ and $h_l$ belong to ${\mathcal U}(X,Y).$ By the same token, it is easy to see that 
${\mathcal U}(X,Y)={\mathbb N}^2,$ which contradicts the fact that $(x_1, x_2, x_3, x_4)$ is a non-trivial solution.
\end{enumerate}
\end{proof}
It is seen from the proof that an uncorrelatedness set of the form $v_j \cup h_k$ may be regarded as maximal in the sense that, if it contains any other point, then it equals ${\mathbb N}^2.$ This fact can be stated as:

\begin{corollary} Let $(j,k)\in{\mathbb N}^2$ be given and $P\notin v_j \cup h_k.$
If $ v_j \cup h_k \cup \{P\} \subseteq {\mathcal U}(X,Y),$ then ${\mathcal U}(X,Y)={\mathbb N}^2$ and, thus, $X$ and $Y$ are independent.
\end{corollary}

Along with Corollary \ref{cor:lp}, this implies that if ${\mathcal U}(X,Y)$ contains $v_j \cup \{(s,t)\},$ then either ${\mathcal U}(X,Y)=v_j\cup h_t$ or 
${\mathcal U}(X,Y)={\mathbb N}^2.$ Similarly, if ${\mathcal U}(X,Y)$ contains $h_k \cup \{(s,t)\},$ then 
either ${\mathcal U}(X,Y)=v_s\cup h_k$ or ${\mathcal U}(X,Y)={\mathbb N}^2.$

The situation with 3-point sets requires a more thorough investigation. The previous theorem shows that no 3 points on the set $v_j \cup h_k$ can form an uncorrelatedness set. In addition, the following holds.

\begin{lemma}\label{lem3} 
If $\{(j,j), (k,k), (l,l)\}\subseteq {\mathcal U}(X,Y),$ where $j,k$ and $l$ are mutually distint, then $L:=\{(m,m):m\in{\mathbb N}\} \subseteq {\mathcal U}(X,Y).$ Moreover, there exist $X, Y\sim \textnormal{Unif}\{a,b,c\}$ such that ${\mathcal U}(X,Y)=L.$
\end{lemma}
\begin{proof}
Consider
\begin{align*}
x_1 + A_j x_2 + A_j x_3 + A_j^2 x_4&=0  \\
x_1 + A_k x_2 + A_k x_3 + A_k^2 x_4&=0 \\
x_1 + A_l x_2 + A_l x_3 + A_l^2 x_4&=0,
\end{align*}
whose general solution is of the form $(0, x_2, -x_2, 0).$ Now, for $x_2\neq 0$ and $(m,n)\in {\mathbb N}^2,$ one has
$$
x_1+A_mx_2+A_nx_3+A_mA_nx_4=(A_m-A_n)x_2=0
$$
if and only if $m=n.$ This completes the proof.
\end{proof}
Next, let us consider a 3-point set where two points are symmetric in the line $L.$ Such a set is not an admissible uncorrelatedness set as the following result shows.

\begin{lemma}
If $\,{\mathcal U}(X,Y)$ contains three distinct points $(j,k),$ $(k,j)$ and $(m,n),$ then $(n,m)$ is also contained in ${\mathcal U}(X,Y).$
\end{lemma}

\begin{proof} Suppose that
\begin{align*}
x_1 + A_{j}x_2 + A_{k}x_3 + A_{j}A_{k}x_4&=0 \\
x_1 + A_{k}x_2 + A_{j}x_3 + A_{k}A_{j}x_4&=0 \\
x_1 + A_{m}x_2 + A_{n}x_3 + A_{m}A_{n}x_4&=0.
\end{align*}
From the first two equations, one concludes that $x_2=x_3.$ Hence, the third equation can be written as
$$
x_1 + A_{m}x_3 + A_{n}x_2 + A_{m}A_{n}x_4=0
$$
which means that $(n,m)\in {\mathcal U}(X,Y).$
\end{proof}

From this point on, for the sake of simplicity in calculations, we will take 
$a=\alpha,$ $b=\alpha \beta$ and $c=\alpha \beta^2$ with $\alpha>0,$ $\beta>1.$
Then $A_j=1+\beta^{-j},$ $j\in{\mathbb N}.$ As it is stated before, we are only interested in distinguishing the trivial and non-trivial solutions of \eqref{syst}. After setting $A_j=1+\beta^{-j},$ with the transformation $y_1=x_4,$ $y_2=x_3+x_4,$ $y_3=x_2+x_4,$ $y_4=x_1+x_2+x_3+x_4,$ the system \eqref{syst} becomes
\begin{align} \label{syst2}
y_1+\beta^j y_2 + \beta^k y_3 + \beta^{j+k} y_4=0.
\end{align}
Due to the nature of this transformation, the trivial and non-trivial solutions of \eqref{syst} and \eqref{syst2} correspond to each other. In other words, $y_i=0,$ $i=1,2,3,4$ if and only if $x_i=0,$ $i=1,2,3,4.$ 
Therefore, ${\mathcal A}$ is an uncorrelatedness set for $X$ and $Y$ if and only if system \eqref{syst2} is satisfied for $(j,k)\in {\mathcal A}$ and violated for $(j,k)\notin {\mathcal A}.$

Despite such discouraging outcomes in search of 3-point uncorrelatedness sets, the next theorem shows such sets actually exist.

\begin{theorem} \label{thm:3pt} There exist $X, Y\sim \textnormal{Unif}\{\alpha,\alpha\beta,\alpha\beta^2\}$ for which
${\mathcal U}(X,Y) = \{(1,3), (2,2), (3,1)\}.$
\end{theorem}

\begin{proof} The system
\begin{align*}
y_1 + \beta y_2 + \beta^3 y_3 + \beta^4y_4=0 \\
y_1 + \beta^2 y_2 + \beta^2 y_3 + \beta^4 y_4=0 \\
y_1 + \beta^3 y_2 + \beta y_3 + \beta^4 y_4=0
\end{align*}
has the general solution $y_1=\beta^4 \gamma, y_2=y_3=0, y_4=-\gamma,$ $\gamma\in {\mathbb R}.$ For $(j,k)\in{\mathbb N}^2,$ the 4-tuple $(\beta^4\gamma, 0,0,-\gamma), \gamma\neq 0,$ is a solution of 
$
y_1 + \beta^j y_2 + \beta^k y_3 + \beta^{j+k} y_4 =0
$  
if and only if $\beta^4-\beta^{j+k}=0,$ which means that $(j,k)\in\{(1,3), (2,2), (3,1)\}.$
\end{proof}

\begin{remark} Theorem \ref{thm:3pt} is true for $\textnormal{Unif}\{a,b,c\}.$ However, the proof is more cumbersome in the general case.
\end{remark}

\begin{theorem} \label{any}
For any integer $m \geq 2,$ there exist $X, Y\sim \textnormal{Unif}\{\alpha,\alpha\beta,\alpha\beta^2\}$ such that 
${\mathcal U}(X,Y) = \{(j,k)\in{\mathbb N}^2: j+k=m\}.$
\end{theorem}

\begin{proof} For $m=2,3$ and $4,$ the statement follows from Proposition \ref{prop1}, Theorem \ref{thmhv} and Theorem \ref{thm:3pt}. Let $m\geqslant 5.$ The 4-tuple $(\beta^m,0,0,-1)$ satisfies the equation
$$
y_1 + \beta^j y_2 + \beta^k y_3 + \beta^{j+k} y_4 =0
$$  
if and only if $j+k=m.$
\end{proof}

\begin{remark}
There exist $X, Y\sim \textnormal{Unif}\{\alpha,\alpha\beta,\alpha\beta^2\}$ with the uncorrelatedness set of $X$ and $Y$ being of any given size $n\in{\mathbb N}_0.$ 
\end{remark}

\section{Some related determinants and their applications} \label{sec:det}

To further proceed, the determinants of special forms need to be investigated some of which are generalizations of the well-known Vandermonde determinant. See, for example \cite{prasolov}. To the best of our knowledge, the determinants presented here have not appeared in the literature.

To start with, let us introduce the notation used in this section. For a non-negative integer $k,$ set 
\begin{align*}
\sigma_k(x,y):=\sum_{i=0}^{k} x^{k-i}y^i = x^k+x^{k-1}y+\cdots+xy^{k-1}+y^k.
\end{align*}
Note that $\sigma_0(x,y)=1$ and $\sigma_k(x,y)=(x^{k+1}-y^{k+1})/(x-y).$ Hence, 
\begin{align}
x^k-y^k=(x-y)\sigma_{k-1}(x,y), \qquad k\geqslant 1. \label{fk}
\end{align}
\begin{lemma} For each $k\in{\mathbb N}_0,$ one has
\begin{align} \label{sig}
\sigma_k(x,y)-\sigma_k(x,z)=(y-z)\sum_{j=0}^{k-1} x^{k-j-1}\sigma_j(y,z).
\end{align}
\end{lemma}
\begin{proof} It is clear that
\begin{align*}
\sigma_k(x,y)-\sigma_k(x,z)&=\sum_{j=0}^k x^{k-j}y^j-\sum_{j=0}^k x^{k-j}z^j =\sum_{j=1}^k x^{k-j}(y^j-z^j).
\end{align*}
Using \eqref{fk}, we obtain
\begin{align*}
\sigma_k(x,y)-\sigma_k(x,z)=\sum_{j=1}^k x^{k-j}(y-z)\sigma_{j-1}(y,z)=(y-z)\sum_{j=0}^{k-1} x^{k-j-1}\sigma_j(y,z).
\end{align*}
\end{proof}
The next lemma gives a result for a $2 \times 2$ determinant.
\begin{lemma} For $0 \leqslant j \leqslant m,$ there holds:
\begin{align} \label{detjm}
\left|
\begin{array}{cc}
\sigma_j(x,y) & \sigma_m(x,y) \\
\sigma_j(x,z) & \sigma_m(x,z)
\end{array}
\right|
=(z-y)\sum_{r=0}^j \sum_{s=j+1}^m x^{j+m-r-s} y^rz^r \sigma_{s-r-1}(y,z).
\end{align}
Here, the double sum represents a symmetric polynomial with positive coefficients.
\end{lemma}
\begin{proof} It can be readily seen that
\begin{align*}
\left|
\begin{array}{cc}
\sigma_j(x,y) & \sigma_m(x,y) \\
\sigma_j(x,z) & \sigma_m(x,z)
\end{array}
\right| &= \sigma_j(x,y)\sigma_m(x,z)-\sigma_j(x,z)\sigma_m(x,y) \\
&=\sum_{r=0}^j \sum_{s=0}^m x^{j+m-r-s} (y^rz^s-z^ry^s).
\end{align*}
Now, the terms corresponding to the grid points $(r,s)=(u,v)$ and $(r,s)=(v,u)$ cancel. As such,
\begin{align*}
\left|
\begin{array}{cc}
\sigma_j(x,y) & \sigma_m(x,y) \\
\sigma_j(x,z) & \sigma_m(x,z)
\end{array}
\right| &=\sum_{r=0}^j \sum_{s=j+1}^m x^{j+m-r-s} (y^rz^s-z^ry^s) \\
&=\sum_{r=0}^j \sum_{s=j+1}^m x^{j+m-r-s} y^rz^r(z^{s-r}-y^{s-r})\\
&=(z-y)\sum_{r=0}^j \sum_{s=j+1}^m x^{j+m-r-s} y^rz^r \sigma_{s-r-1}(y,z).
\end{align*}
\end{proof}
Here comes a generalization of the $4\times 4$ Vandermonde determinant.
\begin{theorem}
Let $1\leqslant m<n$ be integers and 
\begin{align*}
F_{m,n}:=F_{m,n}(x,y,z,t)=\left|
\begin{array}{cccc}
1 & x & x^m & x^n \\
1 & y & y^m & y^n \\
1 & z & z^m & z^n \\
1 & t & t^m & t^n
\end{array}
\right|.
\end{align*}
Then,
\begin{multline*}
F_{m,n}=(y-x)(z-x)(t-x)(z-y)(t-y)(t-z) \times \\ 
\sum_{j=0}^{m-2}\sum_{k=0}^{n-m-1}x^{n-3-j-k}\sum_{r=0}^j \sum_{s=j+1}^{m+k-1}y^{m+j+k-r-s-1} z^r t^r \sigma_{s-r-1}(z,t).
\end{multline*}
\end{theorem}

\begin{proof}
Applying the operations $-x^nC_3+C_4,$ $-x^mC_1+C_3,$ $-xC_1+C_2,$ successively, and using \eqref{fk}, one gets 
\begin{align*}
F_{m,n}&=\left|
\begin{array}{ccc}
y-x & y^m-x^m & y^m(y^{n-m}-x^{n-m}) \\
z-x & z^m-x^m & z^m(z^{n-m}-x^{n-m}) \\
t-x & t^m-x^m & t^m(t^{n-m}-x^{n-m})
\end{array}
\right| \\
&=(y-x)(z-x)(t-x)\left|
\begin{array}{ccc}
1 & \sigma_{m-1}(x,y) & y^m\sigma_{n-m-1}(x,y) \\
1 & \sigma_{m-1}(x,z) & z^m\sigma_{n-m-1}(x,z) \\
1 & \sigma_{m-1}(x,t) & t^m\sigma_{n-m-1}(x,t)
\end{array}
\right|.
\end{align*}
Performing the operations $-R_1+R_2$ and $-R_1+R_3,$ then expanding along the first column, we obtain
\begin{align*}
F_{m,n}&=(y-x)(z-x)(t-x)\left|
\begin{array}{cc}
\sigma_{m-1}(x,z)-\sigma_{m-1}(x,y) & z^m\sigma_{n-m-1}(x,z)-y^m\sigma_{n-m-1}(x,y) \\
\sigma_{m-1}(x,t)-\sigma_{m-1}(x,y) & t^m\sigma_{n-m-1}(x,t)-y^m\sigma_{n-m-1}(x,y)
\end{array}
\right|.
\end{align*}
Using \eqref{sig}, we write
\begin{align*}
F_{m,n}&=(y-x)(z-x)(t-x)\left|
\begin{array}{cc}
(z-y)\sum_{j=0}^{m-2}x^{m-2-j}\sigma_j(y,z) & \sum_{k=0}^{n-m-1}x^{n-m-1-k}(z^{m+k}-y^{m+k}) \\
(t-y)\sum_{j=0}^{m-2}x^{m-2-j}\sigma_j(y,t) & \sum_{k=0}^{n-m-1}x^{n-m-1-k}(t^{m+k}-y^{m+k})
\end{array}
\right|\\
&=(y-x)(z-x)(t-x)(z-y)(t-y)\sum_{j=0}^{m-2}\sum_{k=0}^{n-m-1}x^{n-3-j-k}\left|
\begin{array}{cc}
\sigma_j(y,z) & \sigma_{m+k-1}(y,z) \\
\sigma_j(y,t) & \sigma_{m+k-1}(y,t)
\end{array}
\right|.
\end{align*}
Finally, with the help of \eqref{detjm}, one arrives at
\begin{multline*}
F_{m,n}=(y-x)(z-x)(t-x)(z-y)(t-y)(t-z) \times \\ 
\sum_{j=0}^{m-2}\sum_{k=0}^{n-m-1}x^{n-3-j-k}\sum_{r=0}^j 
\sum_{s=j+1}^{m+k-1}y^{j+m+k-1-r-s} z^r t^r \sigma_{s-r-1}(z,t).
\end{multline*}
\end{proof}
Yet another generalization of the Vandermonde determinant is as follows.
\begin{theorem}\label{Gmn}
Let $1\leqslant m<n$ be two integers and 
\begin{align*}
G_{m,n}:=G_{m,n}(x,y,z,t)=\left|
\begin{array}{cccc}
1 & x^m & x^n & x^{m+n} \\
1 & y^m & y^n & y^{m+n} \\
1 & z^m & z^n & z^{m+n} \\
1 & t^m & t^n & t^{m+n}
\end{array}
\right|.
\end{align*}
Then,
\begin{multline*}
G_{m,n}=(y-x)(z-x)(t-x)(t-y)(z-y)(t-z) \times \\ 
\sum_{k=m}^{n-1} \sum_{j=0}^{m-1} \sum_{p=0}^{m-1} \sum_{s=k-p}^{n-1}\sum_{r=0}^{k-j-1} 
x^{2m+n-3-k-p-j}y^{n+k-2-r-s}z^{j+r}t^{j+r}\sigma_{p+s-j-r-1}(z,t).
\end{multline*}
\end{theorem}

\begin{proof}
Applying the operations $-x^mC_3+C_4,$ $-x^nC_1+C_3,$ $-x^mC_1+C_2,$ successively, and using \eqref{fk}, we get 
\begin{align*}
G_{m,n}&=\left|
\begin{array}{ccc}
y^m-x^m & y^n-x^n & y^n(y^m-x^m) \\
z^m-x^m & z^n-x^n & z^n(z^m-x^m) \\
t^m-x^m & t^n-x^n & t^n(t^m-x^m)
\end{array}
\right| \\
&=(y-x)(z-x)(t-x)\left|
\begin{array}{ccc}
\sigma_{m-1}(x,y) & \sigma_{n-1}(x,y) & y^n\sigma_{m-1}(x,y) \\
\sigma_{m-1}(x,z) & \sigma_{n-1}(x,z) & z^n\sigma_{m-1}(x,z) \\
\sigma_{m-1}(x,t) & \sigma_{n-1}(x,t) & t^n\sigma_{n-1}(x,t)
\end{array}
\right|.
\end{align*}
Performing the operation $-y^nC_1+C_3$ and expanding along the third column, we obtain
\begin{align*}
G_{m,n}&=(y-x)(z-x)(t-x)\left|
\begin{array}{ccc}
\sigma_{m-1}(x,y) & \sigma_{n-1}(x,y) & 0 \\
\sigma_{m-1}(x,z) & \sigma_{n-1}(x,z) & (z^n-y^n)\sigma_{m-1}(x,z) \\
\sigma_{m-1}(x,t) & \sigma_{n-1}(x,t) & (t^n-y^n)\sigma_{m-1}(x,t)
\end{array}
\right|\\
&= (y-x)(z-x)(t-x) \left\{
(t^n-y^n)\sigma_{m-1}(x,t)
\left|
\begin{array}{cc}
\sigma_{m-1}(x,y) & \sigma_{n-1}(x,y)\\
\sigma_{m-1}(x,z) & \sigma_{n-1}(x,z)
\end{array}
\right| \right. \\ &\left.\qquad -
(z^n-y^n)\sigma_{m-1}(x,z)
\left|
\begin{array}{cc}
\sigma_{m-1}(x,y) & \sigma_{n-1}(x,y) \\
\sigma_{m-1}(x,t) & \sigma_{n-1}(x,t) 
\end{array}
\right| \right\}.
\end{align*}
Using \eqref{fk} and \eqref{detjm}, we find
\begin{align*}
G_{m,n}&=(y-x)(z-x)(t-x) \left\{
(t-y)\sigma_{m-1}(x,t)\sigma_{n-1}(y,t)
\left|
\begin{array}{cc}
\sigma_{m-1}(x,y) & \sigma_{n-1}(x,y)\\
\sigma_{m-1}(x,z) & \sigma_{n-1}(x,z)
\end{array}
\right| \right. \\ & \left. \qquad -
(z-y)\sigma_{m-1}(x,z)\sigma_{n-1}(y,z)
\left|
\begin{array}{cc}
\sigma_{m-1}(x,y) & \sigma_{n-1}(x,y) \\
\sigma_{m-1}(x,t) & \sigma_{n-1}(x,t) 
\end{array}
\right|\right\}\\
&=(y-x)(z-x)(t-x)(t-y)(z-y) \times \\
&\qquad \left\{\sigma_{m-1}(x,t)\sigma_{n-1}(y,t) \sum_{j=0}^{m-1} 
\sum_{k=m}^{n-1} x^{m+n-2-j-k} y^j z^j \sigma_{k-j-1}(y,z) \right. \\
&\qquad \left.-\sigma_{m-1}(x,z)\sigma_{n-1}(y,z) \sum_{j=0}^{m-1} 
\sum_{k=m}^{n-1} x^{m+n-2-j-k} y^j t^j \sigma_{k-j-1}(y,t)\right\} \\
&=(y-x)(z-x)(t-x)(t-y)(z-y)\sum_{j=0}^{m-1} \sum_{k=m}^{n-1} x^{m+n-2-j-k} y^j \Phi_{j,k}
\end{align*}
where
\begin{align*}
\Phi_{j,k}&=\sigma_{m-1}(x,t)\sigma_{n-1}(y,t)z^j \sigma_{k-j-1}(y,z)-
\sigma_{m-1}(x,z)\sigma_{n-1}(y,z)t^j \sigma_{k-j-1}(y,t)\\
&=\sum_{p=0}^{m-1} \sum_{s=0}^{n-1}\sum_{r=0}^{k-j-1} x^{m-1-p} y^{n-2-s+k-j-r} (z^{j+r}t^{p+s}-t^{j+r}z^{p+s}).
\end{align*}
Therefore,
\begin{multline*}
G_{m,n}=(y-x)(z-x)(t-x)(t-y)(z-y) \times \\ 
\sum_{k=m}^{n-1} \sum_{j=0}^{m-1} \sum_{p=0}^{m-1} \sum_{s=0}^{n-1}\sum_{r=0}^{k-j-1} 
x^{2m+n-3-k-p-j}y^{n+k-2-s-r}(z^{j+r}t^{p+s}-t^{j+r}z^{p+s}).
\end{multline*}
Now, dividing the sum over $s$ into two parts, one has
\begin{align*}
G_{m,n}&=(y-x)(z-x)(t-x)(t-y)(z-y) \times \\ & \qquad 
\sum_{k=m}^{n-1} \sum_{j=0}^{m-1} \sum_{p=0}^{m-1} \sum_{s=0}^{k-p-1}\sum_{r=0}^{k-j-1} 
x^{2m+n-3-k-p-j}y^{n+k-2-s-r}(z^{j+r}t^{p+s}-t^{j+r}z^{p+s})\\
&\quad +(y-x)(z-x)(t-x)(t-y)(z-y) \times \\ & \qquad 
\sum_{k=m}^{n-1} \sum_{j=0}^{m-1} \sum_{p=0}^{m-1} \sum_{s=k-p}^{n-1}\sum_{r=0}^{k-j-1} 
x^{2m+n-3-k-p-j}y^{n+k-2-s-r}(z^{j+r}t^{p+s}-t^{j+r}z^{p+s}).
\end{align*}
In the first sum, $j+r$ takes on the values $j,j+1,\ldots, k-1$ and $p+s$ takes on the values $p,p+1,\ldots, k-1.$
Since the pair $(k,p)$ runs over the integer nodes of the square $[0,m-1]\times [0, m-1],$ for each fixed $(k^*, p^*)$ there will be a counterpart $(p^*, k^*)$ and the counterparts cancel each other. Thus, the first sum vanishes and it remains
\begin{multline*}
G_{m,n}=(y-x)(z-x)(t-x)(t-y)(z-y) \times \\  \qquad 
\sum_{k=m}^{n-1} \sum_{j=0}^{m-1} \sum_{p=0}^{m-1} \sum_{s=k-p}^{n-1}\sum_{r=0}^{k-j-1} 
x^{2m+n-3-k-p-j}y^{n+k-2-s-r}(z^{j+r}t^{p+s}-t^{j+r}z^{p+s}).
\end{multline*}
Now, it is clear that in each term we have $j+r<p+s.$ Thus,
\begin{multline*}
G_{m,n}=(y-x)(z-x)(t-x)(t-y)(z-y)(t-z) \times \\  \qquad 
\sum_{k=m}^{n-1} \sum_{j=0}^{m-1} \sum_{p=0}^{m-1} \sum_{s=k-p}^{n-1}\sum_{r=0}^{k-j-1} 
x^{2m+n-3-k-p-j}y^{n+k-2-r-s}z^{j+r}t^{j+r}\sigma_{p+s-j-r-1}(z,t).
\end{multline*}
\end{proof}
As an application, the following result is achieved on uncorrelatedness sets.
\begin{theorem} \label{thm1} Let $X,Y\sim \textnormal{Unif}\{\alpha,\alpha\beta, \alpha\beta^2\}.$ If there exists $m\neq 1$ such that
$\{(j_i, mj_i):i=1,2,3,4\}\subseteq{\mathcal U}(X,Y),$ then $X$ and $Y$ are independent.
\end{theorem}
\begin{proof} If $(j_i, mj_i)\in{\mathbb N}^2,$ then $m$ should be a rational number, say $m=b/a.$ Since $m\neq 1,$ we have $a\neq b.$ Due to Remark \ref{rem:W}, without loss of generality, one may take $a<b.$ Since, by Theorem \ref{Gmn}, 
$G_{a,b}(\beta^{j_1/b},\beta^{j_2/b},\beta^{j_3/b},\beta^{j_4/b}) \neq 0,$ the system 
$$
y_1 + \beta^{j_i} y_2 + \beta^{mj_i} y_3 + \beta^{(1+m)j_i} y_4=0, \quad i=1,2,3,4 
$$
has only the trivial solution, implying that this system becomes true for all $(j,k)\in {\mathbb N}^2.$ Hence, ${\mathcal U}(X,Y)={\mathbb N}^2,$ meaning that $X$ and $Y$ are independent. 
\end{proof}

This theorem shows that no 4 points on the line $y=mx,$ $m\neq 1,$ with positive integer coordinates may form an uncorrelatedness set for such random variables. The sharpness of this result is demonstrated in the next statement.

\begin{theorem} \label{thm2} Given integer $m\geqslant 2,$ let $\beta_0=\beta_0(m)$ be the unique solution of $\beta^{m+1}-\beta^2-\beta-1=0$ on $(1, \infty).$ Then, for every $\beta \geqslant \beta_0,$ there exist 
$X, Y\sim \textnormal{Unif}\{\alpha,\alpha\beta,\alpha\beta^2\}$ for which 
${\mathcal U}(X,Y) = {\mathcal A}:=\{(1,m),(2,2m),(3,3m)\}.$
Nevertheless, if $k>4m,$ there exists $\beta\in(1,\beta_0)$ such that $\{(4,k)\} \cup {\mathcal A} \subseteq {\mathcal U}(X,Y) \neq {\mathbb N}^2$ for some $X, Y\sim \textnormal{Unif}\{\alpha,\alpha\beta,\alpha\beta^2\}.$
\end{theorem}

\begin{proof} If $m\geqslant 2,$ then the general solution of the system
\begin{align*}
y_1 + \beta y_2 + \beta^{m} y_3 + \beta^{m+1} y_4 &=0 \\
y_1 + \beta^2 y_2 + \beta^{2m} y_3 + \beta^{2m+2} y_4 &=0 \\
y_1 + \beta^3 y_2 + \beta^{3m} y_3 + \beta^{3m+3} y_4 &=0
\end{align*}  
is
\begin{align*}
y_1&=(\beta^m-\beta)\beta^{2m+2}\gamma, &  y_2&=(1-\beta^{m+1})\beta^{2m}\gamma,\\ 
y_3&=(\beta^{m+1}-1)\beta^2\gamma, & y_4&=(\beta-\beta^{m})\gamma,  \quad \gamma\in{\mathbb R}.
\end{align*}

For $\gamma\neq 0,$ let such a solution satisfy \eqref{syst2}. Then, 
\begin{align*}
D(j,k):=(\beta^{m}-\beta)(\beta^{2m+2}-\beta^{j+k})+(\beta^{m+1}-1)(\beta^{k+2}-\beta^{j+2m})=0.
\end{align*}
Note that
\begin{align*}
&D(1,k)=(\beta^2-1)\beta^{m+1}(\beta^k-\beta^m) =0 \Leftrightarrow k=m, \\
&D(2,k)=\beta^2(\beta-1)(\beta^m+1)(\beta^k-\beta^{2m}) =0 \Leftrightarrow k=2m, \\
&D(3,k)=\beta^2(\beta^2-1)(\beta^k-\beta^{3m}) =0 \Leftrightarrow k=3m. 
\end{align*}
Let $j\geqslant 4.$ 
Since the coefficient of $\beta^j$ in $D(j,k)$ is negative, one has 
$D(j,k)\leqslant D(4,k)$ for all $j \geqslant 4.$
Meanwhile, $D(4,k)=(1-\beta)\beta^2 P(\beta)$ where
\begin{align*}
P(\beta)=(\beta^{m+1}-\beta^2-\beta-1) \beta^{k}
+(\beta^{m+2}+\beta^{m+1}+\beta^m-\beta)\beta^{2m}.
\end{align*}
As the polynomial $\beta^{m+1}-\beta^2-\beta-1$ is increasing on $(1,\infty),$ it has a unique root $\beta_0,$ which is in the interval $(1,2).$ Hence, $D(4,k)<0$ for all positive integers $k$ whenever $\beta \geqslant \beta_0.$ 
Therefore, $D(j,k)$ vanishes if and only if $(j,k)\in{\mathcal A},$ which means that ${\mathcal A}$ can serve as an uncorrelatedness set. 

For the second part of the claim, let $k>4m$ be fixed. Obviously,
$P(1)=0$
and $P(\beta_0)>0.$
Moreover,
$P'(1) = 8m-2k <0.$
Therefore, there exists $\beta^*\in(1,\beta_0)$ for which $D(4,k)=0.$ As a result, $(4,k)\in{\mathcal U}(X,Y).$
For dependent $X$ and $Y,$ if ${\mathcal A}\subseteq {\mathcal U}(X,Y),$ then
$(4,4m)$ cannot be in ${\mathcal U}(X,Y)$ as claimed by Theorem \ref{thm1}.
\end{proof}

\begin{remark} The statement holds for any 3 points on the line $y=mx, m\neq 1.$ However, although the idea of the proof remains the same, the calculations appear to be significantly more complicated. For this reason, the proof is not presented here. The result shows that any 3 points on the line $y=mx, m\neq 1,$ form an uncorrelatedness set in contrast to the cases of a vertical line, a horizontal line and the first bisector by Theorem \ref{thm3p} (i), (ii) and Lemma \ref{lem3}, respectively.
\end{remark}

\section{Concluding remarks and open problems}\label{sec:op}

The outcomes of the present paper show that, even in the case when distributions of $X$ and $Y$ are rather simple, such as  uniform distributions on three points, the description of admissible uncorrelatedness sets may be quite tedious. This, however, depends on the values assumed by the random variables, as Theorem \ref{thm2} shows convincingly.  

Among relatively easy cases, it would be illustrative to consider $X$ and $Y$ uniformly distributed on $\{-\alpha, 0, \alpha\}$ for $\alpha>0.$ Then, one has $A_j=1-(-1)^j.$ Clearly, the parity of $j$ and $k$ matters, thus inspiring the consideration of lattices ${\mathcal A}_1=2{\mathbb N} \times 2{\mathbb N},$ ${\mathcal A}_2=2{\mathbb N} \times (2{\mathbb N}+1),$ 
${\mathcal A}_3=(2{\mathbb N}+1) \times 2{\mathbb N}$ and ${\mathcal A}_4=(2{\mathbb N}+1) \times (2{\mathbb N}+1).$
As the 4-tuple $(1,0,0,0)$ does not satisfy the system \eqref{syst} for any $(j,k)\in{\mathbb N}^2,$ the empty set is a possible uncorrelatedness set.
It occurs that there are no finite uncorrelatedness sets for $X,Y\sim\textnormal{Unif}\{-\alpha, 0, \alpha\}$ other than the empty set, which is by no means similar to the case treated by Theorem \ref{any}. 
Furthermore, for such distributions, if ${\mathcal A}_i \cap {\mathcal U}(X,Y) \neq \emptyset,$ then ${\mathcal A}_i \subseteq {\mathcal U}(X,Y)$ for $i=1,2,3,4.$ In fact, it is not difficult to show that the only admissible uncorrelatedness sets are $\emptyset,$ ${\mathcal A}_i,$ ${\mathcal A}_i\cup {\mathcal A}_j,$ ${\mathcal A}_i\cup {\mathcal A}_j\cup {\mathcal A}_k$ and ${\mathbb N}^2.$ This completely describes all feasible uncorrelatedness sets for such distributions.

In distinction, if $X,Y\sim\textnormal{Unif}\{1,2,3\},$ then the identification of uncorrelatedness sets leads to the  investigation of the $4\times 4$ determinants, whose $i$th rows are $[1 \ A_{j_i} \ A_{k_i} \ A_{j_i}A_{k_i}],$ where $A_j=(3^j-1)/(3^j-2^j).$ The necessary and sufficient condition for such determinants to vanish is a challenging open problem. 
Even more so is the case of $\textnormal{Unif}\{a,b,c\},$ which generates much more complicated settings. 
The results of the current paper fall in between these two ventures and initiate a first attempt to address new problems related to uncorrelatedness sets of discrete distributions.

\section*{Acknowledgments}

The authors express their appreciation to Mr. P. Danesh, Atilim University Academic Writing and Advisory Centre, for improving the presentation of the paper.

\end{document}